\documentclass[12pt,fleqn]{article}

\usepackage{amsmath}
\usepackage{amsthm}
\usepackage{amssymb}
\usepackage{amsfonts}
\usepackage{graphicx}
\usepackage{hyperref}
\usepackage{color}

\newtheorem{lem}{Lemma}

\newtheorem{thm}{Theorem}

\newtheorem{cor}{Corollary}

\newtheorem{prop}{Proposition}

\theoremstyle{remark}
\newtheorem{rmk}{Remark}

\theoremstyle{definition}

 \DeclareMathOperator\re{{Re}}
 \numberwithin{equation}{section}

\newcommand{\D}{\displaystyle}

\numberwithin{equation}{section}

\newcounter{comment}
\begin{document}

\title{Painlev\'{e} III asymptotics of Hankel determinants\\ for a singularly perturbed Laguerre weight}

\author{Shuai-Xia Xu$^a$, Dan Dai$^b$ and Yu-Qiu Zhao$^c$\footnote{Corresponding author (Yu-Qiu Zhao).
 {\it{E-mail
address:}} {stszyq@mail.sysu.edu.cn} }}
  \date{
 {\it{$^a$Institut Franco-Chinois de l'Energie Nucl\'{e}aire, Sun Yat-sen University, GuangZhou
510275,  China}}\\
{\it{$^b$Department of Mathematics, City University of Hong Kong, Tat Chee Avenue, Kowloon, Hong Kong}}\\
 {\it{$^c$Department of Mathematics, Sun Yat-sen University, GuangZhou
510275, China}}
}

\maketitle

\begin{abstract}

In this paper,  we consider the Hankel determinants associated with the singularly perturbed Laguerre weight
$w(x)=x^\alpha e^{-x-t/x}$,  $x\in (0, \infty)$,  $t>0$ and $\alpha>0$.  When the
matrix size $n\to\infty$, we obtain an asymptotic formula for  the Hankel determinants,   valid uniformly for $t\in (0, d]$,  $d>0$ fixed.
A particular  Painlev\'{e} III transcendent is    involved in the  approximation, as well as in the large-$n$ asymptotics  of the leading coefficients and recurrence coefficients for the corresponding perturbed Laguerre polynomials. The derivation is based on the asymptotic results in an earlier paper of the authors,    obtained  by using
the Deift-Zhou nonlinear steepest descent   method.

\end{abstract}


\vspace{5mm}



\section{Introduction and statement of results}

Let $w(x)$ be the following singularly perturbed Laguerre weight
\begin{equation} \label{weight-of-the-paper}
    w(x)=w(x;t, \alpha)=x^{\alpha}e^{-V_t(x)},~~~x\in (0, \infty),~~t>0,~~\alpha>0
\end{equation}
with
\begin{equation} \label{vt-def}
  V_t(x):= x + \frac{t}{x}, \qquad x \in (0,\infty), \quad t >0.
\end{equation}
In this paper, we consider the Hankel determinant associated with the above weight function
\begin{equation} \label{hakel-def}
  D_n[w;t]:= \det (\mu_{j+k})_{j,k=0}^{n-1},
\end{equation}
where $\mu_j$ is the $j$-th moment of $w(x)$, i.e.
\begin{equation}
  \mu_j := \int_0^\infty x^j w(x) dx.
\end{equation}
It is well-known that Hankel determinants are closely related to partition functions in random matrix theory. Indeed, let $Z_n(t)$ be the partition function associated with the weight function in \eqref{weight-of-the-paper}
\begin{equation} \label{par-def}
  Z_n(t):=\int_{0}^{+\infty} \cdots \int_{0}^{+\infty} \prod_{1 \leq j < k \leq n} (\lambda_i - \lambda_j)^2 \lambda_j^\alpha \exp\biggl(-\sum_{j=1}^n V_t(\lambda_j) \biggr) d\lambda_1 \cdots d\lambda_n.
\end{equation}
Then, there is only a constant difference between $D_n[w;t]$ and $Z_n(t)$,  that is
\begin{equation} \label{dn-zn}
  D_n[w;t] = \frac{1}{n!}Z_n(t);
\end{equation}
for example, see \cite[Eq.(1.2)]{ci}. When $t=0$, $w(x)$ is reduced to the classical Laguerre weight. The corresponding partition function $Z_n(0)$ is associated with the Laguerre unitary ensemble and given explicitly below
\begin{equation} \label{zn0}
  Z_n(0) = \prod_{j=1}^n j! \ \Gamma(j + \alpha);
\end{equation}
see \cite[p.321]{mehta}.

The matrix model and Hankel determinants $D_n[w;t]$ associated with the weight function $w(x)$ in \eqref{weight-of-the-paper} were first considered by  Osipov and Kanzieper \cite{Osi:Kan} in bosonic replica field theories. Later,  Chen and Its \cite{ci}  consider this problem again from another point of view, where their  motivation partially originates from an integrable quantum field theory at finite temperature. When $n$ is fixed, they showed that the Hankel determinant $D_n[w;t]$ is the isomonodromy  $\tau$-function of a Painlev\'e III equation multiplied by two factors, that is,
\begin{equation}
  D_n[w;t] = C \cdot \tau_n(t) \, e^{\frac{t}{2}} \; t^{\frac{n(n+\alpha)}{2}},
\end{equation}
where $C$ is a certain constant; see Eq. (6.14) in \cite{ci}.

In this paper, we focus on the asymptotics of Hankel determinants $D_n[w;t]$ as the matrix size $n$ tends to infinity.
In recent years, there has been a considerable amount of interest in the study of asymptotics of Hankel determinants and Toeplitz determinants,  due to their important applications in various branches of applied mathematics and mathematical physics, such as graphical enumeration, one-dimensional gas of impenetrable bosons, two-dimensional Ising model, etc.  For more information, one may refer to \cite{bleher2005,em2003,Zhao:Cao:Dai} and  \cite{Dei:Its:Kra survey},  and references therein.

Note that the asymptotic analysis of $D_n[w;t]$ is totally nontrivial, and it is difficult to derive asymptotic properties about $D_n[w;t]$ from results in Chen and Its \cite{ci}. The main difficulty in the asymptotic study comes from the fact that $w(x)$ has an essential singularity at the origin. For regular $w(x)$, there are quite a lot of asymptotic results about the corresponding matrix models and Hankel determinants in the literature. For example, if $V(x)$ is real analytic and satisfies appropriate boundary conditions, universality results for the limiting kernels have been obtained by Deift {\it{et al.}} \cite{dkmv1}. A decade later, Lubinsky \cite{Lubinsky} improved one of the main results in \cite{dkmv1} by relaxing the condition and requiring $w(x)$ to be continuous only.
Cases when the weight function has singularities, such as   jump discontinuities and algebraic singularities,  are also studied; cf.,  e.g.,    \cite{Its:Kra2008,xz2011} and \cite{ik1,xz2013b}.
It is worth mentioning several  recent work  on Fisher-Hartwig singularities
 \cite{Cla:Its:Kra,Dei:Its:Kra,Krasovsky}.
In \cite{Cla:Its:Kra,Dei:Its:Kra},
emergence of a Fisher-Hartwig singularity and a transition between the two different types of asymptotic behavior for Toeplitz determinants are investigated.
More recently, in \cite{Cla:Kra},   the asymptotics of Toeplitz determinants with merging singularities are considered,  and Painlev\'{e} transcendents are applied to described the transition asymptotics.  The reader is  also referred to the  comprehensive survey paper \cite{Dei:Its:Kra survey} for   historic  background, updated  results and interesting  applications in physics models in this regard.

The asymptotic study of matrix models with an essential singularity was first done by Mezzadri and Mo \cite{Mez:Mo} and Brightmore  {\it{et al.}} \cite{Bri:Mez:Mo} when they are considering asymptotic properties of the partition function associated with the following weight
\begin{equation} \label{Mo-weight}
  w(x;z,s) = \exp\left( -\frac{z^2}{2x^2} + \frac{s}{x} - \frac{x^2}{2} \right), \quad z \in \mathbb{R} \setminus \{0\}, \ 0\leq s < \infty, \ x \in \mathbb{R}.
\end{equation}
Here $w(x):=w(x;z,s)$ has an essential singularity at $x=0$. As pointed out in \cite{Bri:Mez:Mo,Mez:Mo},
although the system of polynomials orthogonal with respect to \eqref{weight-of-the-paper} and \eqref{Mo-weight} can be  mapped to each other when $\alpha=\pm\frac{1}{2}$ and $s=0$, the relation between the respective partition functions is still unclear.
In \cite{Bri:Mez:Mo}, Brightmore  {\it{et al.}} showed that a phase transition emerges as the matrix size $n \to \infty$ and $s,z =O(1/\sqrt{n})$,
 which is characterized by  a Painlev\'e III transcendent.

Recently,  in several different areas of mathematics and physics,   matrix models arise whose weight function has an essential singularity like \eqref{weight-of-the-paper}; see, e.g., Berry and Shukla \cite{Berry:Shukla} in the study of statistics for zeros of the Riemann zeta function, Lukyanov \cite{Lukyanov} in a calculation of finite temperature expectation values in integrable quantum field theory, and \cite{Bro:Fra:Bee,Mez:Sim,Tex:Maj} in the study of the Wigner time delay,
which stands for the average time that an electron spends when scattered by an open cavity.
The Wigner delay time plays  a very important role  in the
theory of mesoscopic quantum dots.
It has been shown that the distribution of the inverse delay times is given by the Laguerre ensemble   \cite{Bro:Fra:Bee,  Tex:Maj}, and the partition function   serves as the moment generating function  \cite{Bri:Mez:Mo,Mez:Sim}.  However,   the distribution of
the Wigner delay time is far from being understood and several interesting questions remain
open \cite{Mez:Sim,Tex:Maj}.

In our previous paper \cite{Xu:Dai:Zhao}, we have studied the eigenvalue correlation kernel $K_n(x,y;t)$ for the unitary matrix ensemble associated with the weight function $w(x)$ in \eqref{weight-of-the-paper}. As mentioned in \cite{Xu:Dai:Zhao}, when $t>0$, the exponent $\frac{t}{x}$ in \eqref{weight-of-the-paper} induces an infinitely strong zero at the origin and changes the eigenvalue distribution near the hard edge $x=0$.  From an orthogonal polynomial point of view, for $t>0$, the model provides a family of non-Szeg\"{o} polynomial; see \cite{cd},  and  \cite{Xu:Dai:Zhao,zz2008}.  The asymptotics are of the Airy type at the soft edge  adjacent to the origin. The reader is also refer to    a relevant discussion in \cite[p. 274]{ci}, and Corollary \ref{cor1} below,  to see that the polynomials with respect to the weight \eqref{weight-of-the-paper} show a singular behavior, as compared with the classic polynomials.

 In \cite{Xu:Dai:Zhao}, we consider the case $t\to 0$, when a phase transition emerges. By applying Deift-Zhou steepest descent method for Riemann-Hilbert (RH) problems and using a double scaling argument, we obtain a new limit for
 the eigenvalue correlation kernel,  which is related to a third-order nonlinear differential equation. This equation is integrable and its Lax pair is given explicitly. Moreover, we have further showed  that this third-order equation is equivalent to a particular Painlev\'e III equation by a simple transformation.  The    transition of the limiting kernel to the Bessel and Airy kernels is  also discussed when the parameter $t$ varies  in a finite interval $(0, d]$.   For more details about the limiting kernel at the hard edge $0$, interested reader may refer to Section 1.2 in \cite{Xu:Dai:Zhao}.

In the present  paper, based on the analysis in  \cite{Xu:Dai:Zhao},
we derive large-$n$ asymptotic formulas for  the Hankel determinants, the leading coefficients and recurrence coefficients,  uniformly for  $t\in (0,d]$, and   their transitions as the parameters $t$ tends to $0^+$ and to a fixed $d>0$. In a sense, the large degree asymptotic expansion for $H_n$ (cf. \eqref{hn-def} and \eqref{hn-asy}) can be interpreted as an asymptotic study of a particular   Painlev\'{e} III transcendent (cf. \eqref{hn-eqn} below).

\subsection{Statement of results}

To state our results, we need to use  a particular solution $r(s)$  to  the following third-order nonlinear differential equation
\begin{equation} \label{r-eqn-introduction}
    2s^2 r' r''' - s^2 {r''}^2 + 2s r' r'' - 4  s {r'}^3 + \left(2r - \frac{1}{4} \right) {r'}^2 + 1=0.
\end{equation}
As pointed out in \cite{Xu:Dai:Zhao}, the above equation is integrable and its Lax pair is given in \cite[Proposition 1.]{Xu:Dai:Zhao}.
It is also shown in \cite[Proposition 2.]{Xu:Dai:Zhao} that by a change of unknown function
$v(s) =sr'(s)$,
the third-order equation for $r(s)$ is reduced to a particular  Painlev\'{e} III  equation for $v(s)$, namely,
  \begin{equation} \label{v-eqn-introduction}
    v''=\frac {{v'}^2}{v}-\frac {v'}{s}+\frac {v^2}{s^2} +\frac \alpha s-\frac 1 v.
  \end{equation}

In the following theorem, we obtain the asymptotic expansion for the Hankel determinant $D_n[w;t]$ associated with the weight in \eqref{weight-of-the-paper} as $n \to \infty$. This expansion is uniformly valid for $0<t\leq d$, where $d$ is a fixed positive constant.

\begin{thm}\label{thm1}
   Let $\alpha>0,t>0$, $w(x)$ be defined in \eqref{weight-of-the-paper} and $D_n[w;0]$ be
   \begin{equation}
     D_n[w;0] = \frac{1}{n!} \prod_{j=1}^n j! \ \Gamma(j + \alpha).
   \end{equation}
   Then as $n\to\infty$, we have the following asymptotic expansion
    \begin{equation} \label{dn-asy}
        D_n[w;t] =  D_n[w;0] \exp\biggl\{ \biggl[1+O\biggl (\frac 1{n^{1/3}}\biggr ) \biggr] \, \int_0^t \frac{1-4\alpha^2 - 8r(2n\xi)}{16 \xi} d\xi\biggr\}
    \end{equation}
   where the error term is uniformly valid for $t\in(0,d],d>0$ and the error term can be improved to $O(1/n)$ if $s=2nt=O(1)$. Here $r(s)$ is a particular solution to the equation \eqref{r-eqn-introduction} which is analytic for $s\in (0, +\infty)$ and satisfies the following boundary conditions
 \begin{equation} \label{r-boundary-behave}
   r(0)=\frac 1 8\left (1-4\alpha^2\right )~~\mbox{and}~~r(s)= \frac 32 s^{\frac 23}-\alpha s^{\frac 13}+O(1)~~\mbox{as}~~s\to +\infty.
 \end{equation}
\end{thm}

\begin{rmk}
    Because \eqref{r-eqn-introduction} is an integrable equation, one may construct a RH problem for $r(s)$ explicitly based on its Lax pair; see Section \ref{sec-parametrix-0} below.
\end{rmk}

Let $\pi_n(x):=\pi_n(x;t)= x^n + \textsf{p}_1(n) x^{n-1} + \cdots$ be monic polynomials orthogonal with respect to the weight function in \eqref{weight-of-the-paper}, that is
\begin{equation} \label{op-pin}
  \int_0^\infty \pi_j(x) \pi_k(x) w(x) dx = h_j \delta_{j,k}.
\end{equation}
Then the Hankel determinant can be given in terms of the constants $h_j$'s in the above formula
\begin{equation}
  D_n[w;t] = \prod_{j=0}^{n-1} h_j.
\end{equation}
It is well-known that orthogonal polynomials satisfy a three-term recurrence relation as follows
\begin{equation} \label{pin-recurrence}
  x\pi_n(x) = \pi_{n+1} (x) + \alpha_n(t) \pi_n(x) + \beta_n(t) \pi_{n-1}(x).
\end{equation}
In \cite{ci}, for fixed $n$, Chen and Its have  showed that the recurrence coefficient $\alpha_n(t)$ satisfies a particular Painlev\'e equation in terms of the parameter $t$. Moreover, based on orthogonal polynomial techniques, they proved that $D_n[w;t]$ is closely related to the isomonodromy  $\tau$-function of a Painlev\'e III equation. We summarize their results in the following theorem.

\begin{thm} \label{ci-thm}{\rm{(Chen and  Its \cite{ci}).}}
  Let $\alpha_n(t)$ be the recurrence coefficient given in \eqref{pin-recurrence},
  \begin{equation} \label{an-def}
    a_n(t) := \alpha_n(t)-(2n+ 1 + \alpha ),
  \end{equation}
  and
  \begin{equation} \label{hn-def}
    H_n(t) := t \frac{d}{dt} \ln D_n[w;t].
  \end{equation}
  Then $a_n(t)$ satisfies the following differential equation
  \begin{equation} \label{an-PIII}
    a_n''(t)=\frac {[a_n'(t)]^2}{a_n}-\frac {a_n'(t)}{t}+(2n+1+\alpha)\frac{a_n(t)^2}{t^2}+\frac{a_n(t)^3}{t^2}+\frac{\alpha}{t}-\frac {1}{a_n(t)},
  \end{equation}
  with the initial conditions
  \begin{equation}\label{an-initial}
    a_n(0) = 0, \qquad a_n'(0) = \frac{1}{\alpha}.
  \end{equation}
  The function $H_n(t)$ also satisfies a second-order nonlinear differential equation
  \begin{equation} \label{hn-eqn}
    [t H_n''(t)]^2 = [n - (2n+\alpha) H_n'(t)]^2 - 4[n(n+\alpha) + t H_n'(t) - H_n(t)]H_n'(t) [H_n'(t) - 1],
  \end{equation}
  with initial conditions
  \begin{equation} \label{hn-ic}
    H_n(0) = 0, \qquad H_n'(0) = -\frac{n}{\alpha}.
  \end{equation}
\end{thm}

\begin{proof}
  See proofs of Theorems 1 and 3 in \cite{ci}. Although the initial conditions \eqref{hn-ic} for $H_n(t)$ are not explicitly given in \cite{ci}, one may derive them from the formulas (3.10)-(3.11) and  (3.21)-(3.22) in that paper, bearing in mind that $b_n(0)=0$ and $\beta_n(0)=n(n+\alpha)$.
\end{proof}

\begin{rmk}
 It has been  pointed out    in Chen and Its \cite{ci},   that \eqref{an-PIII} is a Painlev\'e III equation,  and \eqref{hn-eqn} is the corresponding Jimbo-Miwa-Okamoto $\sigma$-form.
\end{rmk}

In deriving  the asymptotics of the Hankel determinant $D_n[w;t]$, we obtain asymptotics for several  quantities associated  with the orthogonal polynomials $\pi_n(x)$ in \eqref{op-pin}.
Consequently, we get a pair of approximating equations of Painlev\'{e} type, as follows:

\begin{thm} \label{thm3}Let $\alpha>0$, $t>0$, and the orthogonal  weight $w(x)$ be defined in \eqref{weight-of-the-paper}. Then
   we have the following uniform asymptotic approximations  for the recurrence coefficients $\alpha_n(t)$ and $\beta_n(t)$ in \eqref{pin-recurrence},  and   for the leading coefficient $\gamma_n$ of the orthonormal polynomial  $p_n(z)$
  \begin{eqnarray}
    a_n(t) &=&  \frac{s r'(s)}{2n}\left (1 + O\left ( {n^{-1/3}}\right )\right ),\label{an-asy}  \\
     H_n(t) &=& - \frac{8 r(s)+4\alpha^2-1}{16}\left (1 + O\left ({n^{-1/3}}\right  )\right ),\label{hn-asy}  \\
    \alpha_n(t) &=& 2n + \alpha+1 + \frac{s r'(s)}{2n} \left (1+ O\left ( {n^{-1/3}}\right )\right ), \label{alpha-n-asy}  \\
    \beta_n(t) &=& n^2 +  \alpha n + \frac{4 \alpha^2-1+  8 r(s) - 8sr'(s) }{16}\left (1 +  O\left ( {n^{-1/3}}\right ) \right ), \label{beta-n-asy}\\
    \gamma_n(t)  &=&
    \frac 1 {\sqrt{n!\,  \Gamma(n+1+\alpha)}}   \left  (1+ \frac { 8r(s) +4\alpha^2 -1} {32n} \left (1 + O\left ({n^{-1/3}}\right ) \right )\right ), \label{gamma-n-asy}\
  \end{eqnarray} where $s=2nt$, and the uniformity is for $t\in (0, d]$, $d>0$ fixed.
   Moreover, substituting  the above into   equations \eqref{an-PIII} and \eqref{hn-eqn},     and picking up  the leading order terms   as $n \to \infty$,  we obtain
  \begin{equation} \label{an-PIII-limit}
    s^2 r'(s) r'''(s) - s^2 r''(s)^2 + sr'(s) r''(s) - sr'(s)^3 -\alpha r'(s) + 1 =0
  \end{equation}
  and
  \begin{equation} \label{hn-eqn-limit}
    {s^2 {r''}(s)^2}- 2 s r'(s)^3 + \frac{8r(s)- 1}{4}{{r'}(s)^2} +  {2\alpha}{r'(s)}-  1  = 0,
  \end{equation}
  respectively.
\end{thm}

\begin{rmk}
It is shown in  \cite[Section 2.2]{Xu:Dai:Zhao} that \eqref{hn-eqn-limit} is obtained by integrating \eqref{an-PIII-limit} once. Therefore, it is of interest to see that the Painlev\'e III equation \eqref{an-PIII} and the Jimbo-Miwa-Okamoto $\sigma$-form \eqref{hn-eqn} are equivalent asymptotically, with an appropriate rescaling of variables, and with appropriate initial values.  Also, it is worth mentioning that the pair of  equations,
 \eqref{an-PIII-limit} and \eqref{hn-eqn-limit},  have played a key role in  \cite{Xu:Dai:Zhao} to justify the equivalence of
 \eqref{r-eqn-introduction} with  a particular Painlev\'e III equation. Indeed, it is readily observed that a combination of  \eqref{r-eqn-introduction} and    \eqref{an-PIII-limit} gives  \eqref{hn-eqn-limit}, and that a change of unknown function $v=sr'$ turns \eqref{an-PIII-limit} into the  Painlev\'e III equation \eqref{v-eqn-introduction}.
\end{rmk}

Since all  the asymptotic formulas in Theorems \ref{thm1} and  \ref{thm3} have uniform error terms for $t\in (0, d]$, it is of interest to consider the transition of the quantities within such a framework, as $t$ varies in the interval. We take as   examples  the recurrence coefficients, and the logarithmic derivative of the Hankel determinant.
\begin{cor} \label{cor1}With the same conditions in Theorem \ref{thm3},  It holds
\begin{equation*} \alpha_n(t)=  2n + \alpha+1 + \frac s {2\alpha n} +  O\biggl(\frac{s^2}n \biggr)    + O\biggl(\frac{1}{n^{4/3}}\biggr),~~ \beta_n(t) =  n^2 +  \alpha n +O\left (s^2\right )  +  O\biggl(\frac{1}{n^{1/3}}\biggr), \end{equation*}
and
\begin{equation*}H_n(t)=-\frac s {2\alpha} +O\left (s^2\right )  +  O\biggl(\frac{1}{n^{4/3}}\biggr), \end{equation*}
as $n\to\infty$ and $s=2nt\to 0^+$.  On the other hand, as $s=2nt\to \infty$, or, more specifically, as $t\to d^-$ and $n\to\infty$,   we have
\begin{equation*}\alpha_n(t)= 2n+\alpha+1+\frac {t^{2/3}}{2^{1/3}}   \frac 1 {n^{1/3}}
+O\left (\frac 1  {n^{2/3}} \right ),\end{equation*}
\begin{equation*}\beta_n(t)=n^2+\alpha n+ 2^{-4/3}  t^{2/3} n^{2/3}+O\left (   n^{1/3}   \right ),\end{equation*}
and\begin{equation*}H_n(t)=-\frac 3 4 (2t)^{2/3} n^{2/3}  +O\left (  n^{1/3} \right )  . \end{equation*}
\end{cor}
\begin{proof}
  The results are justified by a combination of \eqref{hn-asy}-\eqref{beta-n-asy}  with the boundary conditions \eqref{r-boundary-behave}. Here use has been made of the fact that $r'(0)=\frac 1 \alpha$, as can be derived from \eqref{an-initial} and \eqref{an-asy}.
\end{proof}

To prove our main results in Theorem \ref{thm1} and \ref{thm3}, we adopt the Deift-Zhou steepest descent method for RH problems. This method has achieved great success in the study of asymptotics for orthogonal polynomials and the corresponding random matrix models, for example, see \cite{Dei:Its:Kra,dkmv1,zxz2011,zz2008}. Since most of the RH analysis has been done in our previous paper \cite{Xu:Dai:Zhao}, we will often refer to that paper. However, to make the current paper  self-contained, we will sketch the RH analysis briefly and list some formulas in Section \ref{sec-RH-analysis} below.  The interested reader  may find more details in \cite{Xu:Dai:Zhao}.

The rest of the paper is arranged as follows. In Section \ref{sec-original-rhp}, we provide a RH problem for the corresponding orthogonal polynomials, as well as  two formulas  relating $a_n(t)$ and $H_n(t)$ with the RH problem.  Such relations  are the starting points of our analysis. In Section \ref{sec-RH-analysis}, we list some main steps and key formulas in the RH analysis. Then the proofs of Theorems \ref{thm1} and \ref{thm3} are given in Section \ref{sec-th1-proof} and Section \ref{sec-thm3-proof}, respectively.


\section{RH problem for orthogonal polynomials} \label{sec-original-rhp}

Consider a $2\times2$ Riemann-Hilbert (RH) problem as follows:
\begin{itemize}
  \item[(Y1)]  $Y(z)$ is analytic in
  $\mathbb{C}\backslash [0,\infty)$;

  \item[(Y2)] $Y(z)$  satisfies the jump condition
  \begin{equation}\label{Y-jump}
  Y_+(x)=Y_-(x) \left(
                               \begin{array}{cc}
                                 1 & w(x) \\
                                 0 & 1 \\
                                 \end{array}
                             \right),
    \qquad x\in (0,\infty),\end{equation} where $w(x)=w(x;t)=x^\alpha e^{-x-t/x}$ is the
    weight function defined  in (\ref{weight-of-the-paper});

  \item[(Y3)]  The asymptotic behavior of $Y(z)$  at infinity is
  \begin{equation}\label{Y-infty}
  Y(z)=\left (I+O\left (  1 /z\right )\right )\left(
                               \begin{array}{cc}
                                 z^n & 0 \\
                                 0 & z^{-n} \\
                               \end{array}
                             \right),\quad \mbox{as}\quad z\rightarrow
                             \infty ;\end{equation}

  \item[(Y4)] The asymptotic behavior of $Y(z)$   at the end points $z=0$ are
 \begin{equation}\label{Y-origin}Y(z)=\left(
                               \begin{array}{cc}
                                O( 1) &  O( 1) \\[0.2cm]
                                O( 1) & O( 1)
                                  \\
                               \end{array}
                             \right),\quad \mbox{as}\quad z\rightarrow
                             0 .\end{equation}

\end{itemize}

According to the significant observation  of Fokas, Its and Kitaev \cite{fik}, the
solution of the above RH problem is given in terms of the monic polynomials,
  orthogonal with respect to $w(x)$. This establishes an
important relation between orthogonal polynomials and
RH problems.

\begin{lem} \label{fik-lemma}
{\rm{
(Fokas, Its and Kitaev  \cite{fik}).}}
The unique solution to the above RH problem for $Y$ is given by
\begin{equation}\label{Y-solution}
Y(z)= \left (\begin{array}{cc}
\pi_n(z)& \frac 1 {2\pi i}
\int_{0} ^{\infty}\frac {\pi_n(s) w(s) }{s-z} ds\\[0.2cm]
-2\pi i \gamma_{n-1}^2 \;\pi_{n-1}(z)& -   \gamma_{n-1}^2\;
\int_{0} ^{\infty}\frac {\pi_{n-1}(s) w(s) }{s-z} ds \end{array} \right ),
\end{equation}
where  $\pi_n(z)$ is the monic polynomial in \eqref{op-pin}, and $\gamma_{n}$ is the leading coefficient of the orthonormal polynomial $p_n(z):=\gamma_{n}\pi_n(z)$
with respect to the weight $w(x)=w(x;t)$;  cf., e.g.,  \cite{deift} and \cite{fik}.

\end{lem}

\begin{proof}
  The proof is based on the Plemelj formula and Liouville's theorem.
\end{proof}

It is worth mentioning that the important quantities $a_n(t)$ and $H_n(t)$ in Theorem \ref{ci-thm} can be expressed in terms of the entries of the RH solution $Y$. These relations are very helpful in our future calculations.

\begin{lem} \label{lem-fixn}
  The following identities hold for $n \in \mathbb{N}$:
  \begin{equation} \label{a-n definition from chen}
     a_n(t)=2\pi i t\gamma_n^2 Y_{11}(0)Y_{12}(0)
  \end{equation}
  and
  \begin{equation} \label{hn definition from chen}
    H_n'(t) = - Y_{12}(0)Y_{21}(0),
  \end{equation}
  where $Y_{ij}(z)$ denotes the (i,j)-entry of the matrix-valued function $Y(z)$.
\end{lem}

\begin{proof}
  From Lemma 2 in Chen and Its \cite{ci}, we have
  \begin{equation} \label{an integral from chen}
    a_n(t)=\frac{t}{h_n} \int_0^{\infty}\frac{\pi_n^2(y)w(y)}{y}dy = t \gamma_n^2 \int_0^{\infty}\frac{\pi_n^2(y)w(y)}{y}dy,
  \end{equation}
  where $h_n$ is the constant in \eqref{op-pin} and $\gamma_n$ is the leading coefficient of the orthonormal polynomial. Note that $\D\frac{\pi_n(y)}{y} = q_{n-1}(y) + \frac{\pi_n(0)}{y}$, where $q_{n-1}(y)$ is a polynomial of degree $n-1$. By the orthogonality of $\pi_n(x)$ in \eqref{op-pin}, the above formula gives us
  \begin{equation}
    a_n(t)=t \gamma_n^2 \int_0^{\infty}\frac{\pi_n(y) \pi_n(0)w(y)}{y}dy.
  \end{equation}
  Comparing the above formula with \eqref{Y-solution} gives us \eqref{a-n definition from chen}.

  To get \eqref{hn definition from chen}, we obtain the following formula from Lemma 2 and Eq. (3.21) in \cite{ci}
  \begin{equation} \label{hn prime integral from chen}
    H_n'(t) = \frac{1}{h_{n-1}} \int_0^{\infty} \frac{\pi_n(y)\pi_{n-1}(y)w(y)}{y}dy = \gamma_{n-1}^2\int_0^{\infty}\frac{\pi_n(y)\pi_{n-1}(y)w(y)}{y}dy.
  \end{equation}
  Using an argument  similar to the derivation of \eqref{a-n definition from chen}, we obtain \eqref{hn definition from chen}.
\end{proof}

\begin{rmk} \label{alternative derivation}
 In Chen and Its \cite{ci},   the integral representation \eqref{an integral from chen} is obtained appealing to a ladder operator technique. We note that \eqref{an integral from chen},  \eqref{hn prime integral from chen}  and  \eqref{beta-n hn relation from chen} below  can be derived alternatively,  using the recurrence relation \eqref{pin-recurrence}  and the specific perturbed Laguerre weight \eqref{weight-of-the-paper}. For example,  using  integration by parts,  \eqref{an integral from chen} can be derived as
\begin{equation*}\alpha_n=\gamma_n^2 \int^\infty_0 x\pi_n^2  w(x) dx=-\gamma_n^2 \int^\infty_0 x\pi_n^2 x^\alpha e^{-\frac t x}  de^{-x}= 2n +1+\alpha +  t  \gamma_n^2\int^\infty_0
 \frac {\pi_n^2  w(x) dx}    x.  \end{equation*}
\end{rmk}


\section{Nonlinear steepest descent analysis} \label{sec-RH-analysis}

In this section, we give a sketch of the nonlinear steepest descent analysis for the RH problem.
In the standard Deift-Zhou analysis, one introduces a sequence of
transformations:
\[ Y \mapsto T \mapsto S \mapsto R,\]
and reduces the original RH problem for $Y$  to a new RH problem for $R$, whose jumps are close to the identity
matrix when $n$ is large. Then $R(z)$ can then be expanded into a Neumann series on the whole complex plane. Since the above transformations are all revertible, the uniform asymptotics of the orthogonal
polynomials in the complex plane are obtained for large polynomial degree $n$ when we trace back. Technique difficulties lie in the construction of the  local parametrix in a neighborhood of the origin $z=0$. The parametrix   possesses  irregular singularity both at infinity and at the origin.

\bigskip

\noindent\textbf{Normalization: $Y \mapsto T$.} The first transformation in the Deift-Zhou steepest descent analysis
is to normalize the large-$z$ behavior of $Y(z)$ in \eqref{Y-infty} to  make $T(z)\sim I$ as $z\to\infty$. To this end, we introduce the following $g$-function
\begin{equation}\label{g-function}
    g(z):=\frac 2{\pi}\int_0^1  \log(z-x) \sqrt{\frac{1-x}{x}} dx, \qquad z \in \mathbb{C} \setminus (-\infty, 1],
\end{equation}
where the branch is chosen such that $\arg(z-x)\in(-\pi,\pi)$. Then the first transformation $Y \mapsto T$ is defined as
\begin{equation}\label{TrsnaformationY-T}
    T(z)= (4n)^{-(n+\frac \alpha 2) \sigma_3}e^{-\frac{n l}{2}\sigma_3}Y(4nz) e^{-n(g(z)-\frac{l}{2})\sigma_3}e^{-\frac {t_n}{8nz}\sigma_3} (4n)^{\frac \alpha 2\sigma_3}
\end{equation}
for $z\in\mathbb{C}\backslash[0,\infty)$, where $l=-2(1+\ln 4)$ is the Lagrange multiplier
  and $\sigma_3$ is the Pauli matrix $\begin{pmatrix}
  1 &0 \\ 0 & -1
\end{pmatrix}$; see \eqref{Pauli-matrix} below. Here $t=t_n$ indicates the dependence of the parameter $t$ on the polynomial degree $n$.

It is easily verified that the large-$z$ behavior of $T(z)$ is normalized such that $T(z) = I + O(z^{-1})$ as $z \to \infty$. Note that, because there is a factor $e^{-\frac{t}{x}}$ in the weight function $w(x)$ in \eqref{weight-of-the-paper}, there is a corresponding $e^{-\frac {t_n}{8nz}\sigma_3}$ term in \eqref{TrsnaformationY-T}. As a consequence, one can easily see from \eqref{Y-origin} and \eqref{TrsnaformationY-T} that $T(z)$
possesses an essential singularity at the origin. This kind of singularity is new in the Riemann-Hilbert analysis and requires a new class of parametrix near the origin; see Section \ref{sec-parametrix-0} below.

\bigskip

\noindent\textbf{Opening of the lens: $T \mapsto S$.} Although the first transformation successfully normalizes the large-$z$ behavior of $T$, the original jump matrix \eqref{Y-jump} for $Y$ becomes more complicated. More precisely, because the function $g(z)$ in \eqref{g-function} is not analytic on the interval $(0,1)$, the jump matrix for $T$ is highly oscillatory on  $(0,1)$ when the polynomial degree $n$ is large. To overcome this difficulty, Deift {\it{et al.}} \cite{dkmv1,dkm} borrowed the ideas from  the classical steepest descent method for integrals and introduced the second transformation $T \mapsto S$. In this transformation,  the original interval $[0,\infty)$ is deformed by opening   lens. Then, the rapidly oscillatory jump matrices for $T$ will be reduced to jump matrices  tending  to the identity matrix exponentially,  except in a neighborhood of $(0,1)$.

Before we introduce the transformation, we need one more auxiliary $\phi$-function as follows
\begin{equation}\label{phi-function}
    \phi(z):=2 \int_0^z\sqrt{\frac{s-1}{s}}ds,~~z\in
    \mathbb{C}\backslash[0,\infty),
\end{equation}
where $\arg z\in(0,2\pi)$, such that the Maclaurin expansion $\phi(z)=  4i\sqrt{z}\left\{ 1-\frac z 6+\cdots  \right \}$ holds for $|z|<1$. From its definition, one can see that $\phi(x)>0$ for $x>1$, $\re \phi(z)<0$ in the lens-shaped domains; cf. Figure \ref{contour-for-S}, and that $\phi_\pm (x)= \pm 2i \int^x_0\sqrt{\frac{1-s} s }ds$, purely imaginary, for $x\in (0,1)$.

\begin{figure}[t]
 \begin{center}
 \includegraphics[width=9cm]{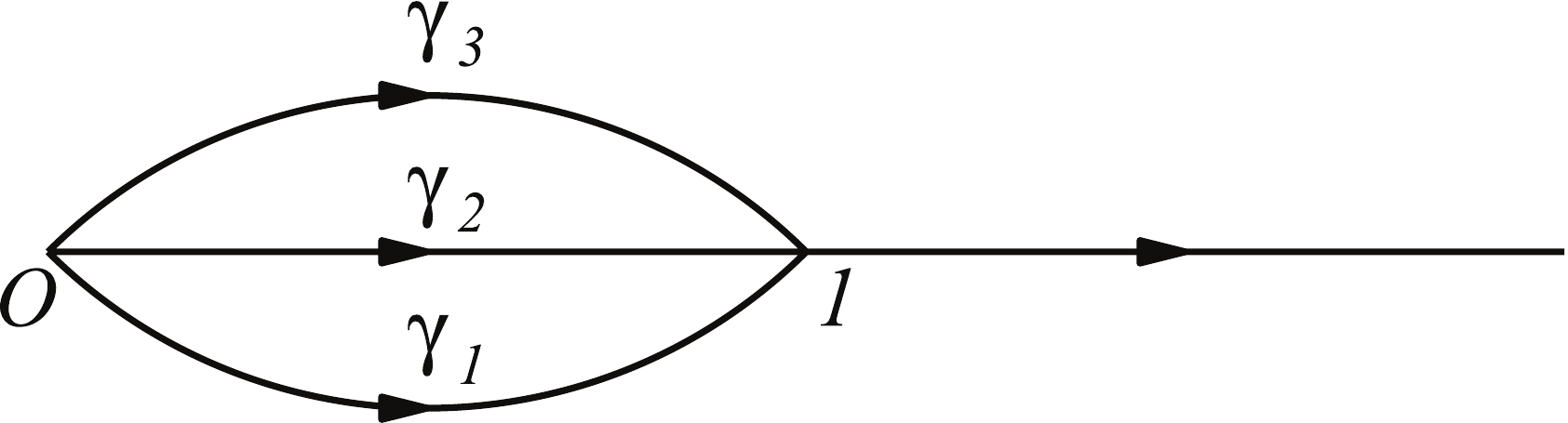}
  \end{center}
  \caption{\small{Contour $\Sigma_{S}$ for the  RH problem for $S(z)$  in the $z$-plane.}}
 \label{contour-for-S}
\end{figure}

The second transformation $T \mapsto S$ is defined as
\begin{equation}\label{transformationT-S}
S(z)=\left \{
\begin{array}{ll}
  T(z), & \mbox{for $z$ outside the lens shaped region;}
  \\  [.4cm]
  T(z) \left( \begin{array}{cc}
                                 1 & 0 \\
                                  - z^{-\alpha}e^{2n\phi(z)} & 1 \\
                               \end{array}
                             \right) , & \mbox{for $z$ in the upper lens
                             region;}\\[.4cm]
T(z) \left( \begin{array}{cc}
                                 1 & 0 \\
                                   z^{-\alpha} e^{2n\phi(z)} & 1 \\
                               \end{array}
                             \right) , & \mbox{for  $z$  in  the  lower
                             lens
                             region, }
\end{array}\right .\end{equation}
where $\arg z\in (-\pi, \pi)$. Then, $S$ satisfies a RH problem with the following jump conditions $J_S(z)$
\begin{equation} \label{S-jump}
  J_S(z)= \begin{cases}
    \left( \begin{array}{cc}
                                 1 & 0 \\
                                  z^{-\alpha} e^{2n\phi(z)} & 1 \\
                               \end{array}
                             \right), &  z\in \gamma_1\cup \gamma_3, \\
    \left( \begin{array}{cc}
                                0&   x^{\alpha} \\
                                   -x^{-\alpha} & 0 \\
                               \end{array}
                             \right), &  z=x\in \gamma_2, \\
    \left( \begin{array}{cc}
                                 1 &  z^{\alpha}e^{-2n \phi(z)} \\
                                  0  & 1 \\
                               \end{array}
                             \right), &   z\in (1,+\infty).
  \end{cases}
\end{equation}
Of course $S(z)$ is still normalized at $\infty$ and possesses an essential singularity at 0.

\subsection{Outside parametrix}

From (\ref{S-jump}), we see  that the jump matrix for $S$ is the identity matrix  plus an exponentially  small  term for
fixed $z\in \gamma_1\cup\gamma_3\cup (1, \infty)$. Neglecting the exponentially small terms, we arrive at an approximating RH problem for $N(z)$, as follows:
\begin{description}
    \item(N1)~~  $N(z)$ is analytic  in  $\mathbb{C}\backslash [0,1]$;

    \item(N2)~~   \begin{equation}\label{N-jump} N_{+}(x)=N_{-}(x)\left(
       \begin{array}{cc}
       0 & x^{\alpha} \\
       -x^{-\alpha} & 0 \\
       \end{array}
       \right)~~~\mbox{for}~~x\in  (0,1);\end{equation}

    \item(N3)~~    \begin{equation}\label{N-infty} N(z)= I+O( 1/ z) ,~~~\mbox{as}~~z\rightarrow\infty .\end{equation}
\end{description}

A solution to the above RH problem can be constructed explicitly,
 \begin{equation}\label{N-solution} N(z)= D_{\infty}^{\sigma_3}M^{-1}    a (z)^{-\sigma_3}   MD(z)^{-\sigma_3}, \end{equation}
where $M=(I+i\sigma_1) /{\sqrt{2}}$, $a
(z)=\left(\frac{z-1} {z}\right)^{1/4}$ with $\arg z\in (-\pi, \pi)$ and $\arg (z-1)\in (-\pi, \pi)$, and the Szeg\"{o} function
 \begin{equation*} D(z)=\left(\frac{z}{\varphi(2z-1)}\right)^{\alpha/2},~~\varphi(z)=z+\sqrt{z^2-1}, \end{equation*}
the branches are chosen such that $\varphi(z) \sim  2z$ as $z\rightarrow\infty$, and  $D_{\infty}=2^{-\alpha}$.

\subsection{Local parametrix $P^{(1)}(z)$ at $z=1$ }

The jump matrices of $SN^{-1}$ are not uniformly close to the unit matrix near the end-points $0$ and $1$, thus local parametrices have to be constructed  in
neighborhoods of the end-points. Near the right end-point $z=1$, we consider a small disk
$U(1,\delta)=\{z~|\;|z-1|<\delta\}$,   $\delta$ being a fixed positive number.
The parametrix $P^{(1)}(z)$ in $U(1,\delta)$ can be constructed in terms of the Airy function
and its derivative as in \cite[(3.74)]{MV}; see also  \cite{deift,dkm}.

\subsection{Local parametrix $P^{(0)}(z)$ at the origin} \label{sec-parametrix-0}

The parametrix, to be constructed  in the neighborhood  $U(0,\delta)=\{z~|\;|z|<\delta\}$  for sufficiently small $\delta$, solves a RH problem as
follows:
\begin{itemize}
  \item[(a)] $P^{(0)}(z)$ is analytic in $U(0,\delta) \backslash  \Sigma_{S}$;

  \item[(b)] In  $U(0,\delta)$, $P^{(0)}(z)$ satisfies the same jump conditions as $S(z)$ does; cf. (\ref{S-jump});

  \item[(c)]  $P^{(0)}(z)$ fulfils the following  matching condition
   on  $\partial U(0,\delta)=\{~z\; |\; |z|=\delta\}$:
    \begin{equation}\label{matchingP0-N}
        P^{(0)}(z)N^{-1}(z)=I+ O\left ( n^{-1/3}\right )~~\mbox{as}~n \to \infty;
    \end{equation}

 \item[(d)] The behavior at the center  $z=0$ is the same as  that of  $S(z)$, which possesses an essential singularity at 0.

\end{itemize}

The construction of the local parametrix is one of the main contribution in our previous paper \cite{Xu:Dai:Zhao}. It involves a model RH problem for $\Psi(\zeta,s)$ which is related to the third-order integrable ODE in \eqref{r-eqn-introduction}. The exact formula for $P^{(0)}(z)$ is given explicitly as follows:
\begin{equation} \label{p0-para}
  P^{(0)}(z)= E(z)\Psi(n^2\phi^2, 2nt_n)e^{-\frac \pi 2 i\sigma_3} (-z)^{-\frac \alpha2\sigma_3}e^{n\phi(z)\sigma_3}, ~~z\in U(0,\delta) \backslash  \Sigma_{S},
\end{equation}
where $E(z)$ is an analytic function in $U(0,\delta)$
\begin{equation}\label{Ez-def}
    E(z)=N(z)e^{\frac \pi 2 i\sigma_3} (-z)^{\frac \alpha 2\sigma_3}\frac{I-i\sigma_1}{\sqrt{2}}
    \left\{n^2\phi^2(z)\right\}^{\frac 1 4\sigma_3}.
\end{equation}
In the above two formulas, we choose $\arg(-z)\in (-\pi, \pi)$ and  $\arg  \left\{ n^2\phi^2(z)\right\}   \in (-\pi, \pi)$.

The important function $\Psi(\zeta)=\Psi(\zeta,s)$ satisfies the following model RH problem

\begin{figure}[t]
 \begin{center}
 \includegraphics[width=9cm]{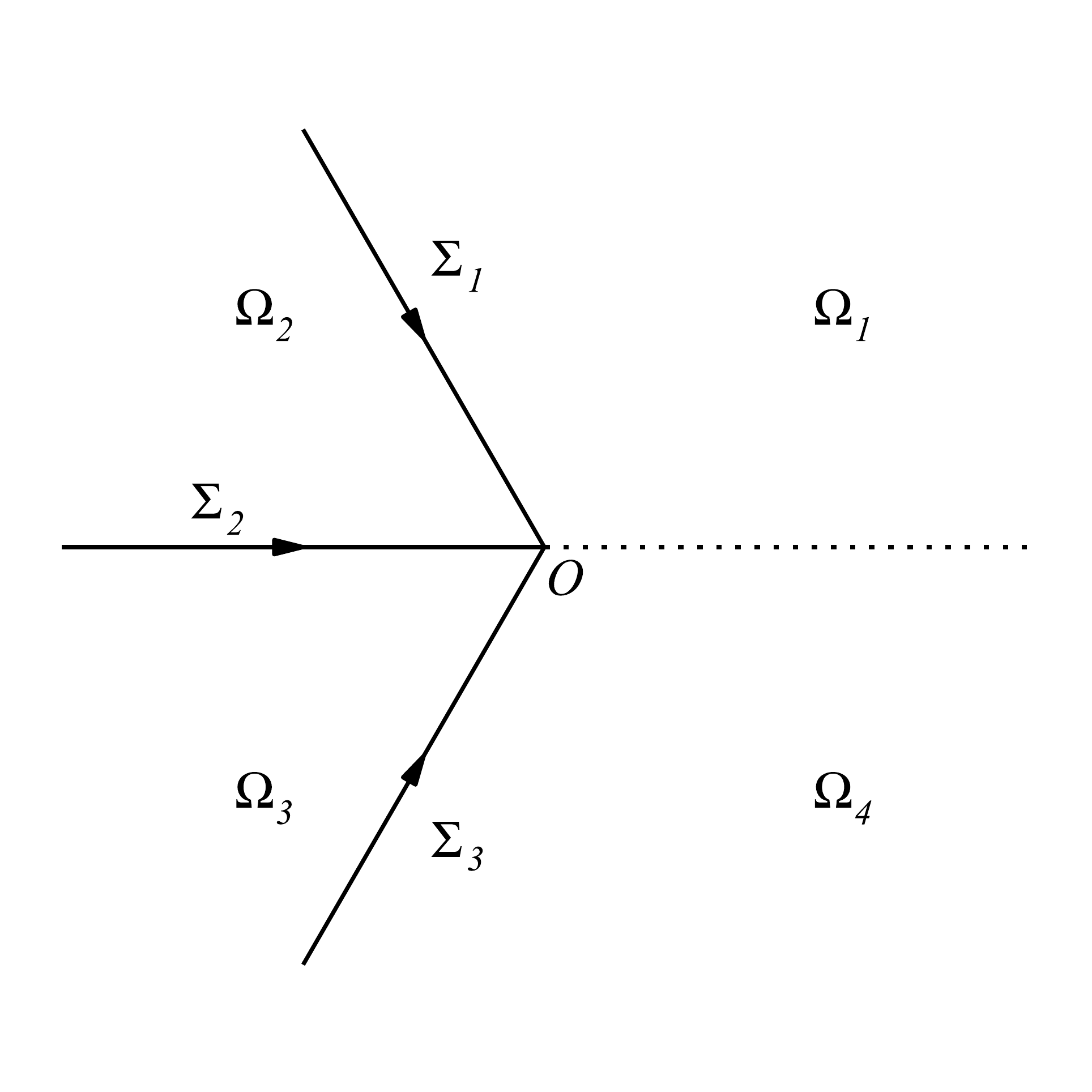} \end{center}
  \caption{\small{Contours and regions  for the model RH problem for $\Psi$  in the $\zeta$-plane, where both  sectors $\Omega_2$ and $\Omega_3$ have an opening angle  $\pi/3$.}}
 \label{contour-for-model}
\end{figure}

\begin{itemize}
  \item[(a)] $\Psi(\zeta)$ is analytic in
  $\mathbb{C}\backslash\cup^3_{j=1}\Sigma_j$, where $\Sigma_j$   are illustrated in Figure \ref{contour-for-model};

  \item[(b)] $\Psi(\zeta)$  satisfies the jump condition
 \begin{equation}\label{Psi-jump}
 \Psi_+(\zeta)=\Psi_-(\zeta)
 \left\{
 \begin{array}{ll}
    \left(
                               \begin{array}{cc}
                                 1 & 0 \\
                               e^{\pi i\alpha}& 1 \\
                                 \end{array}
                             \right), &  \zeta \in \Sigma_1, \\[.4cm]
    \left(
                               \begin{array}{cc}
                                0 &1\\
                               -1&0 \\
                                 \end{array}
                             \right),  &  \zeta \in \Sigma_2, \\[.4cm]
    \left(
                               \begin{array}{cc}
                                 1 &0 \\
                                  e^{-\pi i\alpha} &1 \\
                                 \end{array}
                             \right), &   \zeta \in \Sigma_3;
 \end{array}  \right .
 \end{equation}

\item[(c)] The asymptotic behavior of $\Psi(\zeta)$   at infinity  is
  \begin{equation}\label{Psi-infty}
    \Psi(\zeta, s)= \left[ I + \frac{C_{1}(s)}{\zeta} +
    O\left(\frac{1}{\zeta^2}\right) \right] \; \zeta^{-\frac{1}{4} \sigma_3} \frac{I + i \sigma_1}{\sqrt{2}} e^{\sqrt{\zeta} \sigma_3},~~\arg \zeta\in (-\pi, \pi),~~\zeta\rightarrow \infty,
   \end{equation}
  where $C_1(s)$ is a matrix independent of $\zeta$;

 \item[(d)] The asymptotic behavior of $\Psi(\zeta)$  at $\zeta=0$ is
  \begin{equation}\label{Psi-origin}
  \Psi(\zeta,s)=Q(s)\left\{I+O(\zeta)\right\} e^{\frac s \zeta\sigma_3}\zeta^{\frac \alpha2\sigma_3}\left\{
  \begin{array}{ll}
    I, &  \zeta \in \Omega_1\cup\Omega_4, \\[.4cm]
   \left(
                               \begin{array}{cc}
                                 1 & 0 \\
                               -e^{\pi i\alpha}& 1 \\
                                 \end{array}
                             \right),  &  \zeta \in \Omega_2, \\[.4cm]
    \left(
                               \begin{array}{cc}
                                 1 &0 \\
                                  e^{-\pi i\alpha} &1 \\
                                 \end{array}
                             \right), &   \zeta \in \Omega_3
 \end{array}  \right .
  \end{equation}
  for $\arg \zeta\in (-\pi, \pi)$, as $\zeta\rightarrow 0$,  where $\Omega_1-\Omega_4$ are depicted in Figure \ref{contour-for-model},  $Q(s)$ is a matrix independent of $\zeta$, such that $\det Q(s)=1$, and $\sigma_j$ are the Pauli matrices, namely,
  \begin{equation}\label{Pauli-matrix}
\sigma_1=\left(
                   \begin{array}{cc}
                     0 &1 \\
                    1 & 0 \\
                   \end{array}
                 \right),   ~~\sigma_2=\left(
                   \begin{array}{cc}
                     0 & -i \\
                    i & 0 \\
                   \end{array}
                 \right)~~\mbox{and}~\sigma_3=\left(
                   \begin{array}{cc}
                     1 & 0 \\
                    0 & -1 \\
                   \end{array}
                 \right).
 \end{equation}
 \end{itemize}

\begin{rmk}
  The relation between the above model RH problem for $\Psi$ and the equation \eqref{r-eqn-introduction} is that, let $r(s) = i (C_{1}(s))_{12}$ with $C_1(s)$ given in \eqref{Psi-infty}, then $r(s)$ satisfies the equation \eqref{r-eqn-introduction}. For a detailed derivation, see \cite[Section 2.1]{Xu:Dai:Zhao}.
\end{rmk}

\begin{rmk}
  Readers who are familiar with Riemann-Hilbert analysis may be a little surprised to see the error term $O(n^{-1/3})$ instead of $O(n^{-1})$ in \eqref{matchingP0-N}. In fact, if $s$ is fixed, the function $r(s)$ is bounded and we do get the $O(n^{-1})$ estimation. As $s= 2nt$, to achieve the uniform results in Theorem \ref{thm3} for $t\in(0,d]$, $d>0$ fixed, an asymptotic study for $r(s)$ as $s\to 0$ and $s\to +\infty$ is needed. In our previous paper \cite{Xu:Dai:Zhao}, by performing an asymptotic study of the model RH problem for $\Psi(x,s)$, we get the desired results in the following proposition. The reason why $O(n^{-1/3})$ appears can be seen from the large-$s$ asymptotic behavior of $r(s)$.
  \begin{prop}(Xu, Dai and Zhao \cite{Xu:Dai:Zhao})  There exists a solution $r(s)$ of \eqref{r-eqn-introduction}, analytic for $s\in(0,\infty)$, with the following boundary conditions
    \begin{equation} \label{r boundary conditions}
    r(0)=\frac 18(1-4\alpha^2) \quad  \mbox{and} \quad r(s)=\frac 32s^{\frac 23}-\alpha s^{\frac 13}+O(1)~ \textrm{as}~s\to +\infty.
    \end{equation}
    \end{prop}
\end{rmk}

\subsection{The final transformation $S\mapsto R$}

Now we bring in the final transformation by defining
\begin{equation}\label{transformationS-R}
R(z)=\left\{ \begin{array}{ll}
                S(z)N^{-1}(z), & z\in \mathbb{C}\backslash \left \{ U(0,\delta)\cup U(1,\delta)\cup \Sigma_S \right \};\\[.1cm]
               S(z) (P^{(0)})^{-1}(z), & z\in   U(0,\delta)\backslash \Sigma_{S} ;  \\[.1cm]
               S(z)  (P^{(1)})^{-1}(z), & z\in   U(1,\delta)\backslash
               \Sigma_{S} .
             \end{array}\right .
\end{equation}
\begin{figure}[t]
 \begin{center}
 \includegraphics[width=9cm]{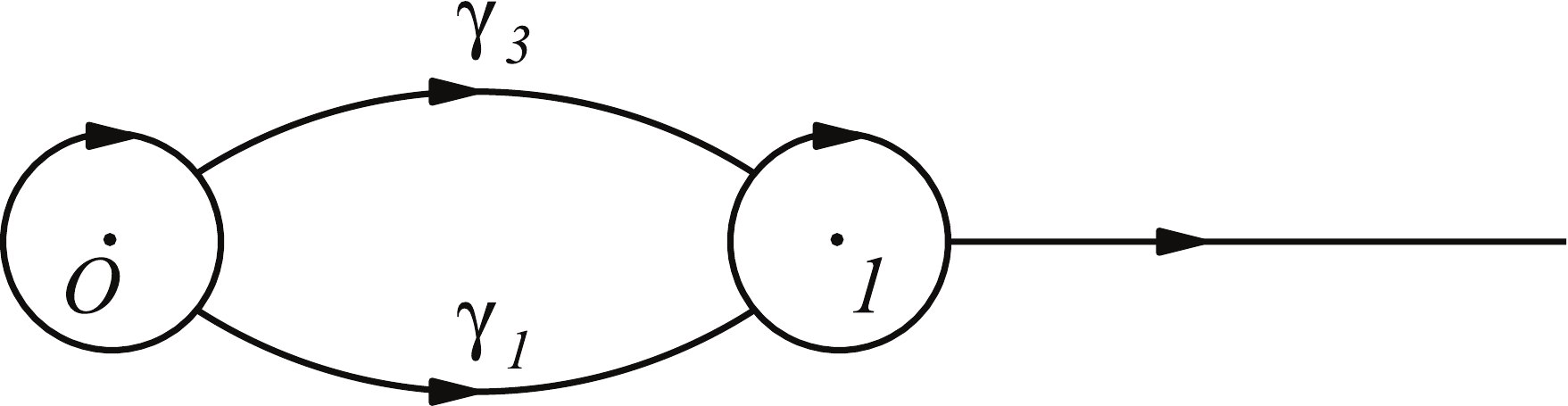} \end{center}
  \caption{\small{Contour $\Sigma_R$ for the  RH problem for $R(z)$  in the $z$-plane.}}
 \label{contour-for-R}
\end{figure}
Then, $R(z)$ solves the following RH problem:
\begin{description}

  \item(R1)~~ $R(z)$ is analytic in $\mathbb{C}  \backslash \Sigma_R$ (see Figure  \ref{contour-for-R} for the contours);

  \item(R2)~~ $R(z)$ satisfies the  jump conditions
  \begin{equation}\label{R-jump}
    R_+(z)=R_-(z)J_{R}(z), ~~z\in\Sigma_R,
  \end{equation}
  where
    $$
    J_R(z)=\left\{ \begin{array}{ll}
                    P^{(0)}(z)N^{-1}(z), ~ &   z\in\partial U(0,\delta),\\[.1cm]
                     P^{(1)}(z)N^{-1}(z),&  z\in\partial
                    U(1,\delta),\\[.1cm]
    N(z)J_S(z)N^{-1}(z), ~& \Sigma_R\setminus \partial
                   ( U(0,\delta)\cup  U(1,\delta));
                 \end{array}\right .$$

  \item(R3)~~  $R(z)$ satisfies the following behavior at infinity:
  \begin{equation}\label{R-infty}
  R(z)= I+ O\left({1}/{z} \right),~~\mbox{as}~ z\rightarrow\infty .
  \end{equation}
\end{description}

It follows from the matching condition (\ref{matchingP0-N}) of the local parametrices and the definition of $\phi(z)$ in \eqref{phi-function} that
\begin{equation}\label{R-jump-approx}J_R(z)=\left\{ \begin{array}{ll}
                     I+O\left (n^{-1/3}\right ),&  z\in\partial
                    U(0,\delta)\cup  U(1,\delta),\\[.1cm]
I+O(e^{-cn}), ~& z\in  \Sigma_R\setminus \partial
                   ( U(0,\delta)\cup  U(1,\delta)),
                 \end{array}\right .
\end{equation}
where $c$ is a positive constant, and the error term is uniform
for $z$ on the corresponding contours. Hence we have
\begin{equation}\label{R-jump-estimate}\|J_R(z)-I\|_{L^2\cap L^{\infty}(\Sigma_R)}=O(n^{-1/3}).
\end{equation}
Then, applying  the now standard  procedure of norm
estimation  of Cauchy operator and using the technique of deformation of contours (cf. \cite{deift,dkm}), it
follows from   (\ref{R-jump-estimate}) that
 \begin{equation}\label{R-approx}R(z)=I+O(n^{-1/3}),
\end{equation} uniformly for $z$ in the whole complex plane, and for $t$ in the interval $(0, d]$, $d>0$ fixed.

This completes the nonlinear steepest descent analysis.

\section{Proof of Theorem \ref{thm1}} \label{sec-th1-proof}

Since we only need to derive one-term asymptotic approximation, it is very helpful to make use of the nice relations between $a_n(t)$, $H_n(t)$ and $Y(0)$ in Lemma \ref{lem-fixn}. These relations can simplify our computations a lot.  If one wishes to obtain more terms in the asymptotic expansions, we need to follow the original ideas in Deift et. al. \cite{dkm}; see also \cite[Sec. 4]{MV}.

\bigskip

\noindent\emph{Proof of Theorem \ref{thm1}.} Tracing back the transformations $R\mapsto S \mapsto T \mapsto Y$, and combining  \eqref{TrsnaformationY-T}, \eqref{transformationT-S} and \eqref{transformationS-R} with \eqref{p0-para}, we have
\begin{equation}\label{Y-trace-back}
    Y_+(4nx)=c_n^{\sigma_3} R(x) E(x) \Psi_- (n^2\phi^2(x), s)  e^{\frac {i\pi} 2   (\alpha-1 )\sigma_3}\left(
    \begin{array}{cc}  1 & 0 \\  1 & 1  \end{array}                                                                                                              \right)\left [w(4nx)\right ]^{-\frac 1 2\sigma_3},
\end{equation}
for $0<x<\delta$, where $   c_n= (-1)^n (4n)^{n+\frac \alpha 2} e^{\frac 1 2 {nl}}. $

To evaluate   $Y(0)$, we need to find out the values of $E(x)$ and $\Psi_- (n^2\phi^2(x), s)$ as $x \to 0$.  From the definitions of $\phi(z)$ and $E(z)$ in \eqref{phi-function} and \eqref{Ez-def}, one can see that
\begin{equation} \label{E-0}
  \phi^2(x) = -16 x + O(x^2) \ \textrm{ and } \  E(x)= 2^{-\alpha\sigma_3} M^{-1} \begin{pmatrix}
    0 & 1 \\ -1 & 0
  \end{pmatrix} (2\sqrt{n})^{\sigma_3} + O(x)
\end{equation}
as $x \to 0$, where $M=(I+i\sigma_1) /{\sqrt{2}}$. Moreover, from the the behavior of $\Psi$ at zero in $\Omega_3$, see \eqref{Psi-origin}, we have
\begin{equation}
  \Psi_- (f_n(x), 2nt)  e^{\frac {i\pi} 2   (\alpha-1 )\sigma_3}\left(        \begin{array}{cc}
                               1 & 0  \\   1 & 1  \end{array}  \right)\left [w(4nx)\right ]^{-\frac 1 2\sigma_3}
    =Q(2nt) (I+O(f_n(x)))d^{\sigma_3},
\end{equation}
where $d=e^{2nt/f_n(x)}e^{\frac {\pi i}2(\alpha-1)}(f_n)^{\alpha/2}w(4nx)^{-\frac 12}$, and  $f_n(z)=n^2   \phi^2(z) $ furnishes a conformal mapping in a neighborhood of $z=0$. Note that the explicit formula for $Q(s)$ can be obtained explicitly as follows
\begin{equation}\label{Q}
Q(s)=\sqrt{\frac {1+q'(s)}2}
\left(    \begin{array}{cc}
 1 & \frac{ir'(s)}{1+q'(s)} \\  i\frac{1-q'(s)}{r'(s)} & 1   \end{array}
 \right)\chi ^{\sigma_3};
\end{equation} see notations and computations in \cite[Sec. 2.1]{Xu:Dai:Zhao}, with $q'(s) =-sr''(s)-\frac 1 2 r'(s)+r(s)r'(s)$,
where  $\chi$   is an arbitrary non-zero factor. Using the fact that $R(x) = I + O(1/n^{1/3})$ and the above formulas, we have
\begin{eqnarray}
  Y_{11}(0)Y_{12}(0)& = & (4n)^{2(n+\frac \alpha2)} e^{nl}\biggl[E(0)Q(2nt)\biggr]_{11}\biggl[E(0)Q(2nt) \biggr]_{12} (1+O  ( {n^{-1/3}}  )) \nonumber \\
  & = & -i 2^{-2\alpha}(4n)^{2(n+\frac \alpha2)} e^{nl} nr'(2nt)(1+O( {n^{-1/3}})) \label{Y-11(0)Y-12(0)}
\end{eqnarray}
and
\begin{eqnarray}
  Y_{12}(0)Y_{21}(0)&=& \biggl[E(0)Q(2nt)\biggr]_{12}\biggl[E(0)Q(2nt) \biggr]_{21}(1+O( {n^{-1/3}})) \nonumber \\
  & =& n r'(2nt)(1+O( {n^{-1/3}})), \label{Y-12(0)Y-21(0)}
\end{eqnarray}
where the error term is uniform for $ t\in(0,d]$, $d>0$.

Combining \eqref{hn definition from chen} and \eqref{Y-12(0)Y-21(0)} gives us
\begin{equation}
  H_n'(t) = - nr'(2nt)(1+O( {n^{-1/3}})).
\end{equation}
Recalling \eqref{r-boundary-behave} and \eqref{hn-ic}, we know that  $r(0) = \frac{1-4\alpha^2}{8}$ and $H_n(0) =0$. Therefore, integrating the above formula from $0$ to $t$ gives us
\begin{equation} \label{asymptotic of h-n}
       H_n(t)=\frac{1-4\alpha^2 - 8r(2nt)}{16}(1+O( {n^{-1/3}})).
\end{equation}
In view of  \eqref{zn0} and \eqref{hn-def}, integrating  one more time of the above formulas gives us \eqref{dn-asy}.

This completes  the proof of Theorem \ref{thm1}. \hfill\qed

\section{Proof of Theorem \ref{thm3}} \label{sec-thm3-proof}

The proof the Theorem \ref{thm3} is similar to that in the previous section.

\bigskip

\noindent\emph{Proof of Theorem \ref{thm3}.} We first derive the formula for $\alpha_n(t)$ in \eqref{alpha-n-asy}. Owing  to the relations in \eqref{an-def} and \eqref{a-n definition from chen}, one needs only to derive the asymptotics for $\gamma_n$ and $Y_{11}(0)Y_{12}(0)$. Note that the asymptotics of $Y_{11}(0)Y_{12}(0)$ have been given in \eqref{Y-11(0)Y-12(0)}.

To get the asymptotics for $\gamma_n$, we need the large-$z$ behavior of $Y(z)$. To achieve it, we again trace back the transformations $R\mapsto S \mapsto T \mapsto Y$. Then, from \eqref{TrsnaformationY-T}, \eqref{transformationT-S} and \eqref{transformationS-R}, we get
\begin{equation} \label{asymptotic of Y}
 Y(4nz)= (4n)^{(n +\frac {\alpha}{2})\sigma_3} e^{\frac{n}{2} l\sigma_3}R(z)N(z)e^{n(g(z)-\frac 1  2l)\sigma_3}e^{\frac {t}{8nz}\sigma_3}(4n)^{-\frac {\alpha}{2}\sigma_3}
\end{equation}
for $z$ is far away from the interval $[0,\infty)$. By the uniform approximation of $R$ in \eqref{R-approx}, we have the following estimate for large $n$ and $|z|$
$$
R(z)=I+O\left (\frac{1}{n^{1/3}z}\right ),
$$
where the error term is uniform for $ t\in(0,d],d>0$. From the explicit formulas of $g(z)$ and $N(z)$ in \eqref{g-function} and \eqref{N-solution}, it is easily seen that
$$
e^{ng(z)\sigma_3}z^{-n\sigma_3}=I-\frac n {4z}\sigma_3 +O\left (\frac{1}{z^2}\right )
$$
and
$$
N(z)=I+2^{-\alpha\sigma_3}\left (-\frac \alpha 4\sigma_3-\frac 1 4 \sigma_2\right )2^{\alpha\sigma_3}\frac 1 z+O\left (\frac{1}{z^2}\right ).
$$
If we expand $Y(4nz)$ in \eqref{asymptotic of Y} as  $ Y(4nz) (4nz)^{-n\sigma_3} =  I + \frac{\tilde Y_1}{z} + O\left (\frac{1}{z^2}\right ) $, then according to Eq. (3.11) in Deift
{\it{et al.}} \cite{dkm}, we have
\begin{equation}
  \gamma_n^2 = \frac{1}{-2\pi i \cdot (4n) \cdot (\tilde Y_1)_{12}}.
\end{equation}
Collecting the above five formulas together, we obtain
\begin{equation} \label{asymptotic of gamma-n}
 \gamma_n^2=\frac 1{\pi}2^{2\alpha+1}(4n)^{-(2n+\alpha+1)}e^{-nl}(1+O(n^{-1/3})),
\end{equation}
where the error term is uniform for $ t\in(0,d],~d>0$. It is readily seen that   \eqref{asymptotic of gamma-n} agrees with  \eqref{gamma-n-asy}. Now combining the above formula and \eqref{Y-11(0)Y-12(0)}, we have
\begin{equation}\label{an asymptotics}
  a_n(t) = 2\pi i t\gamma_n^2 Y_{11}(0)Y_{12}(0) =  t r'(2nt) \left [ 1 + O( {n^{-1/3}}) \right ].
\end{equation}
Thus, \eqref{alpha-n-asy} follows the above formula and \eqref{an-def}.


A full proof of \eqref{gamma-n-asy} can be obtained as follows. Note that $h_n = \gamma^{-2}_n(t)$ and $\frac{d}{dt} w(x;t) = -\frac{1}{x} w(x)$. Then, taking derivative with respective to $t$ on both sides of the orthogonality condition \eqref{op-pin}, we get
\begin{equation}\label{log derivative gamma}
 \frac {\gamma_n'(t)}{\gamma^3_n(t)}=\frac 12\int_0^{\infty} \frac{\pi_n^2(y)w(y)}{y}dy.
\end{equation}
This, together with \eqref{an integral from chen}, gives us
\begin{equation}\label{an and gamma}
  2t \frac {d}{dt}\ln(\gamma_n(t))=a_n(t).
\end{equation}
Recalling \eqref{an asymptotics} and integrating the above formula, we have
\begin{equation}\label{log an asy}
  \ln (\gamma_n(t))-\ln(\gamma_n(0))=\frac{r(s)-r(0)}{4n} (1+O(\frac 1{n^{1/3}})).
\end{equation}
Note that $\gamma_n(0)$ is the leading coefficients of the orthonormal Laguerre polynomials, which can be computed explicitly as follows
\begin{equation}\label{leading coeff laguerre}
 \gamma_n(0)=\sqrt{\frac{D_n[w,0]}{D_{n+1}[w,0]}}=\frac {1}{\sqrt{n!\Gamma(n+1+\alpha)}};
\end{equation}
see \eqref{zn0}. Finally, substituting \eqref{leading coeff laguerre} into \eqref{log an asy} yields \eqref{gamma-n-asy}, where the initial value of $r(0)$ in \eqref{r boundary conditions} is also used.

The asymptotic formula \eqref{beta-n-asy} for $\beta_n(t)$ follows from the asymptotic formula for $H_n(t)$ in \eqref{asymptotic of h-n} and Eq. (3.22) in \cite{ci} given below
\begin{equation}  \label{beta-n hn relation from chen}
  \beta_n(t) = n(n+\alpha) + t H_n'(t) - H_n(t).
\end{equation}
Finally, we substitute the asymptotic formulas \eqref{alpha-n-asy} and \eqref{asymptotic of h-n} into the equations \eqref{an-PIII} and \eqref{hn-eqn},  and then collect the leading terms. The equations \eqref{an-PIII-limit} and \eqref{hn-eqn-limit} follow immediately.

This completes the proof of Theorem \ref{thm3}. \hfill\qed


\section*{Acknowledgements}
 The work of Shuai-Xia Xu  was supported in part by the National
Natural Science Foundation of China under grant number
11201493, GuangDong Natural Science Foundation under grant number S2012040007824, Postdoctoral Science Foundation of China under Grant No.2012M521638, and the Fundamental Research Funds for the Central Universities under grand number 13lgpy41.
Dan Dai was partially supported by a grant from City University of Hong Kong (Project No. 7004065) and a grant from the Research Grants Council of the Hong Kong Special Administrative Region, China (Project No. CityU 11300814).
 Yu-Qiu Zhao  was supported in part by the National
Natural Science Foundation of China under grant number
10871212.



\begin{thebibliography}{99}

    \bibitem{Berry:Shukla} M.V. Berry and P. Shukla, Tuck's incompressibility function: statistics for zeta zeros and eigenvalues, \emph{J. Phys. A}, \textbf{41} (2008),   385202.

    \bibitem{bleher2005} P.M. Bleher and A.R. Its, Asymptotics of the partition function of a random matrix model, \emph{Ann. Inst. Fourier}, \textbf{55} (2005),   1943--2000.

    \bibitem{Bri:Mez:Mo} L. Brightmore, F. Mezzadri and M.Y. Mo, A matrix model with a singular weight and Painlev\'e III,  \emph{ Comm. Math. Phys.},  accepted,  arXiv:1003.2964,
    DOI 10.1007/s00220-014-2076-z.

    \bibitem{Bro:Fra:Bee} P.W. Brouwer, K.M. Frahm  and C.W.J. Beenakker, Quantum mechanical time-delay matrix in chaotic scattering, \emph{Phys. Rev. Lett.}, \textbf{78} (1997), 4737--4740.

  \bibitem{cd} Y. Chen and D. Dai,  Painlev\'{e} V and a Pollaczek-Jacobi type orthogonal polynomials, {\it{J. Approx. Theory}}, {\bf{162}} (2010),  2149--2167.

    \bibitem{ci} Y. Chen and A. Its, Painlev\'{e} III and a singular linear statistics in Hermitian random matrix ensembles. I,  \emph{J. Approx. Theory},  \textbf{162} (2010), 270--297.

    \bibitem{Cla:Its:Kra} T. Claeys, A. Its and I. Krasovsky, Emergence of a singularity for Toeplitz determinants and Painlev\'e V, \emph{Duke Math. J.}, \textbf{160} (2011),   207--262.

    \bibitem{Cla:Kra}
T. Claeys and I. Krasovsky,   Toeplitz determinants with merging singularities,  arXiv:1403.3639.


    \bibitem{deift} P. Deift, {\it{Orthogonal polynomials and random matrices: a Riemann-Hilbert approach}}, Courant Lecture Notes 3, New York University, 1999.

    \bibitem{Dei:Its:Kra} P. Deift, A. Its and I. Krasovsky, Asymptotics of Toeplitz, Hankel, and Toeplitz+Hankel determinants with Fisher-Hartwig singularities, \emph{Ann. of Math.}, \textbf{174} (2011),   1243--1299.

    \bibitem{Dei:Its:Kra survey} P. Deift, A. Its and I. Krasovsky,  Toeplitz matrices Toeplitz determinants under the impetus of the Ising model. Some history and some recent results, {\it{Comm. Pure Appl. Math.}}, \textbf{66} (2013),  1360--1438.

    \bibitem{dkmv1} P. Deift, T. Kriecherbauer, K.T.-R. McLaughlin, S. Venakides and X. Zhou, Uniform asymptotics for polynomials orthogonal with respect to varying exponential weights and applications to universality
    questions in random matrix theory, {\it{Comm. Pure Appl. Math.}}, {\bf{52}} (1999), 1335--1425.

    \bibitem{dkm}  P. Deift, T. Kriecherbauer, K. T-R McLaughlin, S. Venakides,
    and X. Zhou, Strong asymptotics of orthogonal polynomials with respect to
    exponential weights, \emph{Comm. Pure Appl. Math.}, \textbf{52} (1999), 1491--1552.

    \bibitem{em2003} N.M. Ercolani and K.T.-R. McLaughlin, Asymptotics of the partition function for random matrices via Riemann-Hilbert techniques and applications to graphical enumeration, \emph{Int. Math. Res. Not.}, {\bf{2003}} (2003),   755--820.

    \bibitem{fik} A.S. Fokas, A.R. Its  and A.V. Kitaev, The isomonodromy approach to matrix models in 2D quantum gravity, {\it{Comm. Math. Phys.}}, {\bf{147}} (1992), 395--430.

    \bibitem{Its:Kra2008} A. Its and I. Krasovsky, Hankel determinant and orthogonal polynomials for the Gaussian weight with a jump, \emph{Integrable systems and random matrices}, 215--247, Contemp. Math., 458, Amer. Math. Soc., Providence, RI, 2008.

    \bibitem{ik1} A.R. Its, A.B.J. Kuijlaars  and J. \"{O}stensson, Critical edge behavior in unitary random matrix ensembles and the thirty fourth Painlev\'{e}
    transcendent, \emph{Int. Math. Res. Not.},   {\bf{2008}} (2008), article ID rnn017.

    \bibitem{Krasovsky} I.V. Krasovsky, Correlations of the characteristic polynomials in the Gaussian unitary ensemble or a singular Hankel determinant, \emph{Duke Math. J.}, \textbf{139}  (2007), 581--619.

    \bibitem{Lubinsky} D.S. Lubinsky, A new approach to universality limits involving orthogonal polynomials, \emph{Ann. of Math.}, \textbf{170} (2009),   915--939.

    \bibitem{Lukyanov} S. Lukyanov, Finite temperature expectation values of local fields in the sinh-Gordon model, \emph{Nucl. Phys. B}, \textbf{612} (2001),   391--412.

    \bibitem{mehta} M.L. Mehta, {\it{Random matrices}},  3rd ed., Elsevier/Academic Press,  Amsterdam, 2004.

    \bibitem{Mez:Mo} F. Mezzadri and M.Y. Mo, On an average over the Gaussian unitary ensemble, \emph{Int. Math. Res. Not.}, {\bf{2009}} (2009),   3486--3515.

    \bibitem{Mez:Sim} F. Mezzadri and N.J. Simm, Tau-function theory of chaotic quantum transport with $\beta=1,2,4$, \emph{Comm. Math. Phys.}, \textbf{324} (2013),  465--513.

   \bibitem{Osi:Kan} V.A. Osipov and E. Kanzieper, Are bosonic replicas faulty? \emph{Phys. Rev. Lett.}, \textbf{99} (2007), 050602.

    \bibitem{Tex:Maj} C. Texier and S.N. Majumdar, Wigner time-delay distribution in chaotic cavities and freezing transition,  \emph{Phys. Rev. Lett.}, \textbf{110} (2013), 250602.

    \bibitem{MV} M. Vanlessen, Strong asymptotics of Laguerre-type orthogonal polynomials and applications in random matrix theory,
    \emph{Constr. Approx.}, \textbf{25} (2007), 125--175.

   \bibitem{Xu:Dai:Zhao} S.-X. Xu, D. Dai and Y.-Q. Zhao, Critical edge behavior and the Bessel to Airy transition in the singularly perturbed Laguerre unitary ensemble, \emph{Comm. Math. Phys.},   \textbf{332} (2014), 1257--1296.

   \bibitem{xz2011}S.-X. Xu and  Y.-Q. Zhao,  Painlev\'{e}  XXXIV asymptotics of orthogonal polynomials for the Gaussian weight with a jump at the edge,
 {\it{Stud. Appl. Math.,}}  {\bf{127}} (2011),  67--105.

    \bibitem{xz2013b}S.-X. Xu and  Y.-Q. Zhao,  Critical edge behavior in the modified Jacobi ensemble and the Painlev\'{e} V transcendents,
   {\it{J. Math. Phys.}}, {\bf{54}} (2013),  083304.



    \bibitem{Zhao:Cao:Dai} Y. Zhao, L.-H. Cao  and D. Dai, Asymptotics of the partition function of a Laguerre-type random matrix model, \emph{J. Approx. Theory}, \textbf{178} (2014), 64--90.

    \bibitem{zxz2011}J.-R.  Zhou, S.-X.  Xu and Y.-Q. Zhao,  Uniform asymptotics of a system of Szeg\"{o} class polynomials via the Riemann-Hilbert approach,  {\it{Anal. Appl.}}, {\bf{9}} (2011),  447--480.

    \bibitem{zz2008}J.-R. Zhou and Y.-Q. Zhao,  Uniform asymptotics of the Pollaczek polynomials via the Riemann-Hilbert approach, {\it{Proc. R. Soc. Lond. Ser. A}}, {\bf{464}} (2008),  2091--2112.

\end{thebibliography}
\end{document}